\documentclass[a4paper,12pt,leqno]{article}

\usepackage[english]{babel}
\usepackage[utf8]{inputenc}
\usepackage[T1]{fontenc}
\usepackage{amsthm}
\usepackage{amsmath}
\usepackage{amssymb}
\usepackage{mathrsfs}
\usepackage[all]{xy}
\usepackage{fmtcount}
\usepackage{titlesec}
\usepackage{tocloft}
\usepackage{titletoc}
\usepackage{tikz}
\usepackage{pdfpages}
\usepackage{mathtools}
\usepackage{appendix}

\usepackage{geometry}
\usepackage{color}
\usepackage{textcomp}
\usepackage{graphicx}
\usepackage{esint}
\usepackage{wrapfig}
\usepackage{float}
\usepackage{enumitem}

\newtheorem{thm}{Theorem}[section]

\newtheorem{prop}[thm]{Proposition}

\newtheorem{lemma}[thm]{Lemma}

\newtheorem{csq}[thm]{Consequence}

\theoremstyle{definition}
\newtheorem{defi}[thm]{Definition}

\newtheorem{ex}[thm]{Example}

\newtheorem{nota}[thm]{Notation}
\newtheorem{rem}[thm]{Remark}
\newtheorem{const}[thm]{Construction}

\theoremstyle{remark}

\newcommand{\pic}{\mathrm{Pic}}

\newcommand{\h}{\mathrm{H}}

\newcommand{\spec}{\mathrm{Spec}}

\newcommand{\res}{\mathrm{R}_{k^\prime/k}}
\newcommand{\ext}{\mathrm{Ext}^1}

\usepackage{titlepic}

\usepackage{hyperref}
\hypersetup{
    colorlinks=true,
    linkcolor=blue,
    citecolor=brown
}

\title{On geometry of pseudo-reductive groups}
\author{Alexandre Lourdeaux}
\date{\today}

\begin{document}

\maketitle

\begin{abstract}
Basing ourselves on the work of Conrad-Gabber-Prasad, we study retract rationality of some pseudo-reductive groups and calculate the Picard groups of some pseudo-semisimple groups.


\paragraph{MSC classe(s) :} 20G15
\end{abstract}


	\section{Introduction}

\paragraph{Context.}

Pseudo-reductive groups generalize reductive groups. A smooth connected linear algebraic group $G$ over a field $k$ is \emph{pseudo-reductive} if $G$ does not contain any non-trivial subgroup defined over $k$ which is smooth, connected, normal and unipotent; in other words, the $k$-unipotent radical of $G$ is the trivial subgroup $1 \subseteq G$. For a smooth connected linear algebraic group, reductivity implies pseudo-reductivity. The converse holds over perfect base fields but is false when the base field is imperfect (\cite[Prop. 1.1.9 \& 1.1.10]{CGP}). Pseudo-reductive groups were first studied by Tits (years 1991-92 and 1992-93 of \cite{Tits}) and then Conrad, Gabber and Prasad pursued the work which led to the books \cite{CGP} and \cite{ConradPrasad_classification}. For an introduction to pseudo-reductive groups following \cite{CGP}, see the Bourbaki talk by Remy \cite{Remy_pseudoreductif}.

In this paper we will be interested in pseudo-reductive groups and \emph{pseudo-semisimple} groups (that is \emph{perfect} pseudo-reductive groups) of a certain form. The groups under study are those which can be expressed as a quotient of the Weil restriction of a reductive/semisimple groups by the Weil restriction of a central subgroup. Two geometric aspects are considered.

First of all, we study retract rationality. It happens that pseudo-reductive groups are not necessarily unirational over a separably closed field and may have infinite Picard groups (\cite[Ex. 5.14]{Rosengarten_picard}). But pseudo-semisimple groups are unirational and their Picard groups are finite and additive (that is, $\pic(G \times G)=\pic(G) \oplus (G)$ for such a group $G$) as those of reductive groups (\cite[Prop. 3.3]{Rosengarten_picard}). Retract rationality was introduced for fields by Saltman in \cite{Saltman} and by considering function fields of varieties one defines retract rationality for varieties. For a variety, retract rationality is weaker than rationality and than stable rationality, but it is stronger than unirationality. In practice, retract rationality is closer to rationality than to unirationality and it is often possible to state results for retract rational varieties primarily stated for rational varieties. However, retract rationality can be checked via a pratical criterion which makes it easier to show. To study rationality questions it is thus interesting to begin by studying retract rationality. This is what we do in Section \ref{section rationalite retractile}. This section is a first step in the study of retract rationality of pseudo-reductive/pseudo-semisimple groups.

In a second part, we study the groups of extensions by the multiplicative group for pseudo-semisimple groups. Since pseudo-semisimple groups are perfect, they are unirational (\cite[Prop. A.2.11]{CGP}), thus by \cite[Prop. 3.1]{Rosengarten_picard} their groups of extensions by $\mathbb{G}_m$ are equal to their Picard groups which are moreover finite. Using \cite[\S 5]{ConradPrasad_classification}, things work the same way as for semisimple groups with their \emph{universal cover}. Along this part, we shall need to determine the character groups of Weil restrictions of diagonalisable groups, so we propose a partial generalization of \cite[Th. II.2.4]{Oesterle} which is to be found in the appendix \ref{section caracteres weil}.

In Section \ref{sous section cas des groupes pseudo-semi-simples}, we are interested in the groups of extensions by the multiplicative group $\mathbb{G}_m$ for pseudo-semisimple groups. For this purpose we use \cite[\S 5]{ConradPrasad_classification} where Conrad and Prasad study central extensions of perfect groups. From \textit{loc. cit.} we derive Proposition \ref{prop groupe extensions pseudo ss} which leads us to two tangible applications : \begin{prop}[{Prop. \ref{prop inv groupes pss}}]
Let $k^\prime/k$ be a finite, purely inseparable field extension. Write $p$ for the characteristic exponent of $k$ and define $h$ to be the least non-negative integer such that $(k^\prime)^{p^h} \subseteq k$. Let $G^\prime$ be a (connected) semisimple, simply connected $k^\prime$-group and let $\mu^\prime$ be a central subgroup of $G^\prime$ assumed either \'etale or diagonalisable. Then, writing $G:=\res(G^\prime)/\res(\mu^\prime)$, one has an isomorphism  \[ \ext(G,\mathbb{G}_m) \simeq p^h \cdot \mathrm{Hom}_{k^\prime-\text{grp}}(\mu^\prime,\mathbb{G}_{m,k^\prime}) .  \]  
\end{prop} and : \begin{thm}[{Th. \ref{prop rigidite pss}}]
Let $G$ be a standard pseudo-semisimple $k$-group which is pseudo-split. For any separable extension $K/k$ (finite or not), the scalar extension homomorphism $\ext(G,\mathbb{G}_m) \to \ext(G_K,\mathbb{G}_m)$ is an isomorphism.
\end{thm}

In Appendix \ref{section caracteres weil} we determine the character group of Weil restrictions $\res(G_{k^\prime})$ where $k^\prime/k$ is finite and purely inseparable, and $G/k$ is an algebraic group of a certain type. This is Theorem \ref{thm caracteres restriction weil}. It is used to complete calculations for Proposition \ref{prop inv groupes pss} and Theorem \ref{prop rigidite pss} from Section \ref{sous section cas des groupes pseudo-semi-simples}.

\paragraph{Content.} In Section \ref{section rationalite retractile} we review the basics of retract rationality following \cite{Merkurjev_retract} and study pseudo-reductive groups of a specific form to show they are retract rational. More precisely, \begin{thm}[{Th. \ref{prop groupe retract rationnel}}]
Let $k^\prime/k$ be a finite, purely inseparable field extension. Let $G^\prime$ be a (connected) split reductive $k^\prime$-group; let $T^\prime$ be a central $k^\prime$-subgroup of $G^\prime$. Then $\mathrm{R}_{k'/k}(G')/\mathrm{R}_{k'/k}(T^\prime)$ is retract rational over $k$.

If $G^\prime$ is not assumed split, then $\mathrm{R}_{k'/k}(G')/\mathrm{R}_{k'/k}(T^\prime)$ is retract rational over a separable closure of $k$.
\end{thm}

\paragraph{Conventions, notation and vocabulary.} The letter $k$ denotes the ground field over which (almost) every algebro-geometric object of the paper is defined. We fix an algebraic closure $\bar{k}$ of $k$, and a sperable closure $k_s \subseteq \bar{k}$.

Let $K$ be a field. By \emph{algebraic $K$-group} we mean a group scheme of finite type over $K$ and by \emph{linear algebraic $K$-group} an affine group scheme of finite type over $K$. The $K$-character group of an algebraic $K$-group $G$ is denoted by $\widehat{G}(K)$. A \emph{$K$-variety} is a separated and geometrically integral $K$-scheme of finite type over $K$. Given a finite field extension $K/k$ and a (quasi-projective) $K$-scheme $X$, $\mathrm{R}_{K/k}(X)$ means the Weil restrion of $X$ through the extension $K/k$ (it is indeed a scheme by \cite[\S 7.6]{BLR}).

We recall some basic definitions from \cite{CGP} to be used below. A \emph{pseudo-reductive} $k$-group is a smooth connected linear algebraic $k$-group which has a trivial unipotent $k$-radical. A \emph{pseudo-semisimple} group is a perfect pseudo-reductive group. A smooth linear algebraic $k$-group is called \emph{pseudo-split} if it contains a maximal $k$-torus which is $k$-split.

\paragraph{Acknowledgments.}  The present paper is part of the author's PhD thesis. It was done at Institut Camille Jordan (Lyon 1 University, France). The author is deeply thankful to Philippe Gille.


	\section{Retract rationality of some pseudo-reductive groups}
	\label{section rationalite retractile}

In this section we review some basics on retract rationality following \cite{Merkurjev_retract}. Then we prove that this property is satisfied by a particular family of pseudo-reductive groups. The base field is $k$.

			\subsection{Retract rational varieties}

\begin{defi}
Let $l/k$ be a field extension. A $k$-variety $X$ is called \emph{$l$-retract rational} or \emph{retract rational over $l$} if there exist \begin{itemize}
\item  a non-empty open subset $U$ of an affine space $\mathbb{A}_l ^n$,
\item a non-empty open subset $V$ of $X_l$ and,
\item two morphisms $r : U \to V$, $s : V \to U$ defined over $l$
\end{itemize} such that $r \circ s = \mathrm{id}_V$.
\end{defi}

For an arbitrary local ring $A$, denote by $\bar{A}$ the residue field of $A$.

\begin{prop}[{\cite[Prop. 3.1]{Merkurjev_retract}}]
\label{prop retract merkurjev}
Let $X$ be a $k$-variety. The following assertions are equivalent :
\begin{enumerate}
\item $X$ is retract rational over $k$;
\item For every local $k$-algebra $A$, there exists a non-empty open subset $V \subseteq X$ such that $V(A) \to V(\bar{A})$ is onto;
\item For every local $k$-algebra with $\bar{A}$ infinite, there exists a non-empty open subset $V \subseteq X$ such that  $V(A) \to V(\bar{A})$ is onto;
\item For every local $k$-algebra $A$ with $\bar{A}\simeq k(X)$ as $k$-algebras, the generic point of $X$ belongs to the image of $X(A) \to X(\bar{A})$.
\end{enumerate}
\end{prop}

Actually, from the proof of \cite[Prop. 3.1]{Merkurjev_retract} or the proof of \cite[Prop. 1.2]{ColliotSansuc_rationality}, one can extract a more specific criterion :
\begin{lemma}
Let $X$ be a $k$-variety. Let $A=k[X_1,\cdots,X_m]_\mathfrak{p}$ be a polynomial $k$-algebra localized in a prime ideal $\mathfrak{p}$ such that $\bar{A} \underset{k-alg}{\simeq} k(X)$. Then $X$ is retract rational over $k$ if, and only if, there exists a non-empty open subscheme $V \subseteq X$ such that $V(A) \to V(\bar{A})$ is onto. \hfill \qed
\end{lemma}

For a $k$-variety, the links between the different rationality properties known for varieties are the following (\cite[Prop. 3.4]{Merkurjev_retract}) : \[ k \text{-rational}  \Rightarrow k \text{-stably rational} \Rightarrow k \text{-retract rational} \Rightarrow k \text{-unirational} .  \]

An instance for which rationality can be replaced by retract rationality is the following proposition. This is a straighforward generalization of \cite[Lem. 6.6]{Sansuc_Brauer}.

\begin{prop}
\label{prop additivite picard}
Let $X$ and $Y$ be two (non-empty) smooth $k$-varieties. Assume that $Y$ is retract rational over a separable closure of $k$ and that $Y(k) \neq \emptyset$. Then the projections $X \times_k Y \to X, \, Y$ induce a group isomorphism \[ \pic(X) \oplus \pic(Y) \overset{\sim}{\longrightarrow} \pic (X \times_k Y) . \]
\end{prop}

\begin{proof} Denote by $\phi_{X,Y}$ the homomorphism from the statement of the proposition.

$\diamond$ Assume first $k$ separably closed. We follow the proof of \cite[Lem. 11]{ColliotSansuc_Requivalence} and just add a few lines. The smooth varieties $X$ and $Y$ have rational points, so it is straighforward to see that $\phi_{X,Y}$ is one-to-one.

Then, let $V$ be a non-empty open subscheme of $Y$ and let $V_0$ be a non-empty open subscheme of an affine space $\mathbb{A}^n _k$ such that there exist $r : V_0 \to V$ and $s : V \to V_0$ satisfying $r \circ s = \mathrm{Id}_V$. Let $U$ be a non-empty affine open subset of $X$. The following diagram is commutative and has exact rows : \[ \xymatrix{
\mathrm{Div}_{X \setminus U} (X) \ar[r] \ar[d] & \mathrm{Pic}(X) \oplus \mathrm{Pic}(Y) \ar[r] \ar[d]^{\phi_{X,Y}} & \mathrm{Pic}(U) \oplus \mathrm{Pic}(Y) \ar[r] \ar[d]^{\phi_{U,Y}} & 0 \\
\mathrm{Div}_{X \times_k Y \setminus U \times_k Y} (X \times_k Y) \ar[r] & \mathrm{Pic}(X \times_k Y) \ar[r] & \mathrm{Pic}(U \times_k Y) \ar[r] & 0
} . \] The left vertical arrow is an isomorphism, so the following equivalence holds : $\phi_{X,Y}$ is onto if, and only if, $\phi_{U,Y}$ is onto. The same argument with $(Y,U)$ instead of $(X,Y)$ shows that $\phi_{X,Y}$ is onto if, and only if, $\phi_{U,V}$ is onto. And with $(U,\mathbb{A}_k^n)$ instead of $(X,Y)$, we see that $\phi_{U,\mathbb{A}^n_k}$ is onto if, and only if, $\phi_{U,V_0}$ is onto. But $\phi_{U,\mathbb{A}^n _k}$ is onto by \cite[Th. 3.11]{Weibel_Kbook}. Thus, from the commutative diagram \[ \xymatrix{
\mathrm{Pic}(U) \oplus \mathrm{Pic}(V) \ar[r]^-{\phi_{U,V}} \ar[d]^{  \mathrm{Id}_U^\ast \oplus r^\ast} & \mathrm{Pic}(U \times_k V) \ar[d]^{( \mathrm{Id}_U \times r)^*} \ar@/^5pc/[dd]^{\mathrm{Id}} \\
\mathrm{Pic}(U) \oplus \mathrm{Pic}(V_0) \ar[r]^-{\phi_{U,V_0}} \ar[d]^{  \mathrm{Id}_U^\ast \oplus s^\ast} & \mathrm{Pic}(U \times_k V_0) \ar[d]^{( \mathrm{Id}_U \times s)^\ast} \\
\mathrm{Pic}(U) \oplus \mathrm{Pic}(V) \ar[r]^-{\phi_{U,V}}  &  \mathrm{Pic}(U \times_k V)
} \] we extract \[ \xymatrix{
\mathrm{Pic}(U) \oplus \mathrm{Pic}(V_0) \ar@{->>}[r] \ar[d] & \mathrm{Pic}(U \times_k V_0) \ar@{->>}[d] \\
\mathrm{Pic}(U) \oplus \mathrm{Pic}(V) \ar[r]  &  \mathrm{Pic}(U \times_k V)
} . \] Thus $\phi_{U,V}$ is onto, and so is $\phi_{X,Y}$.

$\diamond$ When $k$ is not assumed separably closed, we proceed just as in the proof of \cite[Lem. 6.6]{Sansuc_Brauer}, replacing the reference to \cite[Lem. 11]{ColliotSansuc_Requivalence} by the previous point.
\end{proof}

\paragraph{The case of algebraic groups.} For an algebraic group, being retract rational can be checked for the whole space and not just for an open subset. This is a convenient fact for proofs. 

\begin{lemma}
\label{lemme retract homogene}
Let $G$ be a smooth connected algebraic $k$-group and let $A$ be a local $k$-algebra. Assume that the set of $k$-points $G(k)$ is dense in $G$ (this is the case when $G$ is $k$-unirational and $k$ is infinite). If there exists a dense open subscheme $V$ of $G$ such that $V(A) \to V(\bar{A})$ is onto, then $G(A) \to G(\bar{A})$ is also onto.
\end{lemma}

\begin{proof}
This is just a homogeneity argument. First of all, $\bigcup_{g \in G(k)} \left( g V(A) \right) \to \bigcup_{g \in G(k)} \left( g V(\bar{A}) \right)$ is onto. But since $G(k)$ is dense in $G$, the union of all of the $gV$ for $g$ lying in $G(k)$ is equal to $G$. Thus, $G(\bar{A})= \bigcup_{g \in G(k)} \left( g V(\bar{A}) \right)$ because $\bar{A}$ is a field, so $\bigcup_{g \in G(k)} \left( g V(A) \right) \to G(\bar{A})$ is onto, but since $\bigcup_{g \in G(k)} \left( g V(A) \right) \subseteq G(A)$ we are done.
\end{proof}

\begin{csq}
\label{retract homogene}
If $k$ is infinite and $G$ is $k$-retract rational, then for every local $k$-algebra $A$, the map $G(A) \to G(\bar{A})$ is onto.
\end{csq}

			\subsection{Examples of retract rational pseudo-reductive groups}

Pseudo-reductive groups are not always retract rational (\cite[Ex. 11.3.1]{CGP}). However, all smooth and connected perfect linear algebraic $k$-groups are $k$-unirational (\cite[Prop. A.2.11]{CGP}) ; this applies in particular to pseudo-semisimple $k$-groups. 

In order to prove Theorem \ref{prop groupe retract rationnel}, the two following lemmas will be needed. By \textit{fppf} we mean the "finitely presented and faithfully flat" Grothendieck topology on a scheme. Flat cohomology groups will be denoted by $\h^d_{\textit{fl}}(-,-)$.

\begin{lemma}[{\cite[Exp. XXIV, Prop. 8.2]{SGA3}}]
\label{lemme pss rr}
Let $S^\prime \to S$ be a scheme morphism. Then for every \textit{fppf} group sheaf $\mathcal{G}^\prime$ on $S^\prime$, the map $\mathrm{H}^1_{\textit{fl}}(S,\mathrm{R}_{S^\prime/S}(\mathcal{G^\prime})) \to \mathrm{H}^1_{\textit{fl}}(S^\prime,\mathcal{G}^\prime)$ is one-to-one, and its image is the set formed by the classes of $\mathcal{G}^\prime$-torsors that become trivial over an \textit{fppf} cover $R \times_S S^\prime$ of $S^\prime$ induced by an \textit{fppf} cover $R$ of $S$.
\end{lemma}

\begin{lemma}
\label{rappel pss rr}
Let $A$ be a semi-local ring. Then every finite $A$-algebra is semi-local. In particular its Picard group is trivial.
\end{lemma}

We will use this lemma in the following context : Let $l/k$ be a finite ring extension and let $A$ be a local $k$-algebra. Then the Picard group of $A \otimes_k l$ is trivial.

\begin{proof}
Assume $A$ is a semi-local ring and let $B$ be a finite $A$-algebra. Replacing $A$ by its quotient with respect to the kernel of $A \to B$, we may assume that $A$ is a subring of $B$. Thus $A \to B$ is an injective and integral ring extension, so the Going-up Theorem implies that \emph{every} maximal ideal of $B$ lies over a (unique) maximal ideal of $A$. Furthermore, since $A \hookrightarrow B$ is a finite ring extension, there are only finitely many maximal ideals of $B$ lying over a given maximal ideal of $A$. The fact that $A$ has only finitely many maximal ideals implies that this is also the case for $B$.

Regarding the Picard group, we just need to recall that projective modules of finite type and of constant rank over any semi-local ring are free by \cite[II.\S 5.3. Prop. 5]{Bourbaki_AC}.
\end{proof}

Here is the main result of this section.

\begin{thm}
\label{prop groupe retract rationnel}
Let $k^\prime/k$ be a finite, purely inseparable field extension. Let $G^\prime$ be a connected split reductive $k^\prime$-group; let $T^\prime$ be a central $k^\prime$-subgroup of $G^\prime$. Then $\mathrm{R}_{k'/k}(G')/\mathrm{R}_{k'/k}(T^\prime)$ is retract rational over $k$.

Generally, if $G^\prime$ is not assumed split, then $\mathrm{R}_{k'/k}(G')/\mathrm{R}_{k'/k}(T^\prime)$ is retract rational over a separable closure of $k$.
\end{thm}

Groups of the form $\mathrm{R}_{k'/k}(G')/\mathrm{R}_{k'/k}(T^\prime)$ are pseudo-reductive by \cite[1.3.4]{CGP}. If $G^\prime$ is semisimple simply connected, then these groups are pseudo-semisimple by \textit{loc. cit.}.

\begin{proof}
Write $G$ for the algebraic $k$-group $\mathrm{R}_{k'/k}(G')/\mathrm{R}_{k'/k}(T^\prime)$ and fix a local $k$-algebra $A$. First we notice that $k$ may be assumed to be infinite because a finite field is perfect, so in the finite case, $k^\prime=k$ and the group $\mathrm{R}_{k'/k}(G')/\mathrm{R}_{k'/k}(T^\prime)$ is just a split reductive group, hence a rational variety (the assumption "$k$ infinite" is useful to invoke Consequence \ref{retract homogene} at some points of the proof). We want to show that $G(A) \to G(\bar{A})$ is onto for every local $k$-algebra $A$. By Proposition \ref{prop retract merkurjev} this will imply that $G$ is retract rational over $k$.

Consider the exact sequence of \textit{fppf} sheaves over $\spec(A)$ and $\spec(\bar{A})$ \[ 1 \to \mathrm{R}_{k'/k} (T^\prime) \to \mathrm{R}_{k'/k}(G') \to G \to 1 . \] Taking \textit{fppf} cohomology yields the following commutative diagram with exact rows \[ \xymatrix{
 \mathrm{R}_{k'/k} (T^\prime)(A) \ar[d]^{\phi_1} \ar[r] & \mathrm{R}_{k'/k}(G')(A) \ar[d]^{\phi_2} \ar[r] & G(A) \ar[d]^{\phi_3} \ar[r] & \mathrm{H}^1_{\textit{fl}}(A,\mathrm{R}_{k'/k} (T^\prime)) \ar[d]^{\phi_4} \ar[r]  & \mathrm{H}^1_{\textit{fl}}(A,\mathrm{R}_{k'/k} (G')) \ar[d]^{\phi_5 '}   \\
\mathrm{R}_{k'/k} (T^\prime)(\bar{A}) \ar[r] & \mathrm{R}_{k'/k}(G')(\bar{A}) \ar[r] & G(\bar{A}) \ar[r] & \mathrm{H}^1_{\textit{fl}}(\bar{A},\mathrm{R}_{k'/k} (T^\prime)) \ar[r]  &   \mathrm{H}^1_{\textit{fl}}(\bar{A},\mathrm{R}_{k'/k} (G'))   \\
} . \] The surjective part of the Five Lemma say that $G(A) \to G(\bar{A})$ is onto if $\phi_2$ and $\phi_4$ are onto and that $\phi_5^\prime$ is one-to-one. Let's prove this.

Actually, we don't need to show that $\phi_5^\prime$ is one-to-one. Indeed, $G^\prime$ is split, so there is a split torus $T_0 \subseteq G^\prime$ such that $T^\prime$ sits in $T_0$. Then, writing $B$ for either $A$ or $\bar{A}$, $\mathrm{H}^1_{\textit{fl}}(B,\mathrm{R}_{k'/k} (T^\prime)) \to \mathrm{H}^1_{\textit{fl}}(B,\mathrm{R}_{k'/k} (G'))$ factorises through $\mathrm{H}^1_{\textit{fl}}(B, \mathrm{R}_{k'/k}(T_0))$. But the latter set can be realized as a subset of $\mathrm{H}^1_{\textit{fl}}(B \otimes_k k', T_0))$ by Lemma \ref{lemme pss rr} and $\mathrm{H}^1_{\textit{fl}}(B \otimes_k k', T_0)$ is a singleton since $T_0/k^\prime$ is a split torus and $B \otimes_k k^\prime$ is a semilocal ring by Lemma \ref{rappel pss rr}. Thus the arrow $\mathrm{H}^1_{\textit{fl}}(B,\mathrm{R}_{k'/k} (T^\prime)) \to \mathrm{H}^1_{\textit{fl}}(B,\mathrm{R}_{k'/k} (G'))$ is the zero map. We are then reduced to apply the Five Lemma to the commutative diagram with exact rows \[ \xymatrix{
 \mathrm{R}_{k'/k} (T^\prime)(A) \ar[d]^{\phi_1} \ar[r] & \mathrm{R}_{k'/k}(G')(A) \ar[d]^{\phi_2} \ar[r] & G(A) \ar[d]^{\phi_3} \ar[r] & \mathrm{H}^1_{\textit{fl}}(A,\mathrm{R}_{k'/k} (T^\prime)) \ar[d]^{\phi_4} \ar[r] & 1 \ar[d]^{\phi_5} \\
\mathrm{R}_{k'/k} (T^\prime)(\bar{A}) \ar[r] & \mathrm{R}_{k'/k}(G')(\bar{A}) \ar[r] & G(\bar{A}) \ar[r] & \mathrm{H}^1_{\textit{fl}}(\bar{A},\mathrm{R}_{k'/k} (T^\prime)) \ar[r] & 1 
} . \]

$\diamond$ First of all, $\phi_5$ is one-to-one.

$\diamond$ Secondly, since $G^\prime$ is a split reductive group, it is rational over $k^\prime$. So its Weil restriction through $k^\prime/k$ is rational over $k$ and is in particular retract rational over $k$. By Consequence \ref{retract homogene} $\phi_2$ must be onto.

$\diamond$ Let's show that $\phi_4$ is onto. 

The $k^\prime$-group $T^\prime$ is diagonalizable. So it is a product of copies of the multiplicative group $\mathbb{G}_{m,k^\prime}$ and of groups of roots of unity $\mu_{l,k'}$ for $l \in \mathbb{N}$. For the $\mathbb{G}_{m,k^\prime}$ factors, $\mathrm{H}^1_{\textit{fl}}(A,\mathrm{R}_{k'/k} (\mathbb{G}_{m,k^\prime}))$ and $\mathrm{H}^1_{\textit{fl}}(\bar{A},\mathrm{R}_{k'/k} (\mathbb{G}_{m,k^\prime}))$ are singletons.

For factors $\mu_{l,k'}$ with $l$ prime to the characteristic exponent $p$ of $k$, the $k^\prime$-group $\mu_{l,k'}$ is \'etale, hence by \cite[Cor. A.5.13]{CGP} the canonical map $\mu_{l,k} \to \res(\mu_{l,k'})$ is an isomorphism, so for all local algebras $B$, \[ \mathrm{H}^1_{\textit{fl}}(B,\mathrm{R}_{k'/k}(\mu_{l,k'}))=\h^1_{\textit{fl}}(B,\mu_{l,k})=\mathbb{G}_m(B)/\mathbb{G}_m (B)^l . \] Since $\mathbb{G}_m$ is $k$-(retract) rational, the map $\mathbb{G}_m(A) \to \mathbb{G}_m(\bar{A})$ is onto. Thus \[ \mathrm{H}^1_{\textit{fl}}(A,\mathrm{R}_{k'/k} (\mu_{l,k'})) \to \mathrm{H}^1_{\textit{fl}}(\bar{A},\mathrm{R}_{k'/k} (\mu_{l,k'})) \]  is onto.

For factors $\mu_{q,k'}$ with $q=p^r$, the following sequence of algebraic $k$-groups is exact by Lemma \ref{suite} below : \begin{equation}
\label{suite exacte retract} 1 \to \mathrm{R}_{k'/k}(\mu_{q,k'}) \to \mathrm{R}_{k'/k}(\mathbb{G}_m) \to \mathrm{R}_{k(k'^{q})/k}(\mathbb{G}_m) \to 1 ,
\end{equation} where the arrow $\mathrm{R}_{k'/k}(\mu_{q,k'}) \to \mathrm{R}_{k'/k}(\mathbb{G}_m)$ is inclusion and the arrow $\mathrm{R}_{k'/k}(\mathbb{G}_m) \to \mathrm{R}_{k(k'^q)/k}(\mathbb{G}_m)$ is elevation to the $q$-th power. Considering the sequence \eqref{suite exacte retract} for $B=A$ and $\bar{A}$ yields a commutative diagram with exact rows \[ \xymatrix{
 (A \otimes_k k(k'^{q}))^\ast \ar[r] \ar[d] & \mathrm{H}^1_{\textit{fl}}(A,\mathrm{R}_{k'/k} (\mu)) \ar[r] \ar[d]^{\phi_4} & \mathrm{H}^1_{\textit{fl}}(A,\mathrm{R}_{k'/k} (\mathbb{G}_m)) \ar[d] \\
(\bar{A} \otimes_k k(k'^q))^\ast \ar[r] & \mathrm{H}^1_{\textit{fl}}(\bar{A},\mathrm{R}_{k'/k} (\mu)) \ar[r] & \mathrm{H}^1_{\textit{fl}}(\bar{A},\mathrm{R}_{k'/k} (\mathbb{G}_m)) \\
 } . \] The groups in the right column are zero by Lemmas \ref{lemme pss rr} and \ref{rappel pss rr}. Moreover, the left vertical arrow is surjective because $\mathrm{R}_{k(k'^{q})/k}(\mathbb{G}_m)$ is $k$-retract rational. Thus we find that $\mathrm{H}^1_{\textit{fl}}(A,\mathrm{R}_{k'/k} (\mu_{q,k'})) \to \mathrm{H}^1_{\textit{fl}}(\bar{A},\mathrm{R}_{k'/k} (\mu_{q,k'}))$ is onto ; that is $\phi_4$ is onto.
\end{proof}

\begin{lemma}
\label{suite}
Let $q=p^r$ be a power of $p$. The sequence of algebraic groups \[ 1 \to \mathrm{R}_{k'/k}(\mu_{q,k'}) \overset{f}{\longrightarrow} \mathrm{R}_{k'/k}(\mathbb{G}_m) \overset{g}{\longrightarrow} \mathrm{R}_{k(k'^{q})/k}(\mathbb{G}_m) \to 1  \] is exact, where $f$ is the obvious homomorphism and $g$ is the homomorphism given by the collection of \[ \left\lbrace \begin{array}{ccc}
(R \otimes_k k^\prime)^\ast & \longrightarrow & (R \otimes_k k(k'^q))^\ast \\
\sum_i a_i \otimes \lambda_i & \longmapsto & \sum_i a_i^q \otimes \lambda_i^q
\end{array} 
\right. \] for all $k$-algebras $R$.
\end{lemma}

\begin{proof}
Let $R$ be a $k$-algebra. The kernel of $g$ is easily seen to be $\mathrm{R}_{k'/k}(\mu_{q,k'})(R)$.

In order to show that $g$ is onto, we show that it is onto on $\bar{k}$-points where $\bar{k}$ is an algebraic closure of $k$. So let $ \sum_i a_i \otimes_k \lambda_i \in  \mathrm{R}_{k(k'^q)/k}(\mathbb{G}_m)(\bar{k})=(\bar{k} \otimes_k k(k'^q))^\ast$. We can assume that the $\lambda_i$'s are $q$-th powers of elements of $k^\prime$, that is $\lambda_i= \kappa_i^q$ for $\kappa_i \in k^\prime$. Then taking $b_i$ to be the $q$-th rooth of $a_i$ for all $i$, we see that $\left( \sum_i a_i \otimes \lambda_i \right)=\left( \sum_i b_i \otimes \kappa_i \right)^q$.
\end{proof}

\begin{ex}
Let $k^\prime/k$ be a finite and purely inseparable field extension. Let $q \neq 1$ be a positive integer and consider the group $G=\res(\mathrm{SL}_{q,k^\prime})/\res(\mu_{q,k^\prime})$. When $q$ is coprime to $p$, this group is actually isomorphic to $\res(\mathrm{PGL}_{q,k^\prime})$ (\cite[Cor. A.5.4(3)]{CGP}), which is rational over $k$. In general, Theorem \ref{prop groupe retract rationnel} says that $G$ is retract rational over $k$.
\end{ex}

	\section{Extension groups of pseudo-semisimple groups}
		\label{sous section cas des groupes pseudo-semi-simples}

We determine the group of extensions by $\mathbb{G}_m$ for some pseudo-semisimple groups. Thanks to \cite[Th. 5.1.3]{ConradPrasad_classification} we can proceed as in the case of semisimple groups via the notion of universal covers.

Note that since pseudo-semisimple groups are perfect, they are unirational (\cite[Prop. A.2.11]{CGP}). Thus, by \cite[Prop. 3.1]{Rosengarten_picard}, their groups of extensions by $\mathbb{G}_m$ are equal to their Picard groups.

\begin{nota}
The $K$-character group of an algebraic $K$-group $G$ is denoted by $\widehat{G}(K)$
\end{nota}

		\subsection{Universal covers of perfect groups}

In \cite{ConradPrasad_classification}, Conrad and Prasad establish a result on \emph{tame} central extensions of smooth connected linear algebraic groups which are perfect. It generalizes the notion of universal cover for semisimple groups.

We sum up \cite[Th. 5.1.3]{ConradPrasad_classification} for the convenience of the reader.

\begin{defi}[{\cite[Def. 5.1.1]{ConradPrasad_classification}}]
\begin{enumerate}
\item An affine group scheme $\mu$ over $k$ is \emph{tame} if $\mu$ has no non-trivial unipotent $k$-subgroup (\cite[Def. 5.1.1]{ConradPrasad_classification}).
\item A central extension of a linear algebraic $k$-group $G$ \[ 1 \to \mu \to H \to G \to 1 \] is said to be \emph{tame} if the group $\mu$ is affine and tame.
\end{enumerate}
\end{defi}

We will be interested in tame extensions \[ 1 \to \mu \to H \to G \to 1 \] of smooth connected perfect linear algebraic groups $G$ where $\mu$ is \emph{central} in $H$ and $H$ is also a smooth, connected and perfect algebraic group.

Recall the following construction : given a smooth and connected linear algebraic group $H$ over a field $K$, if the radical $\bar{R}$ of $H_{\bar{K}}$ is defined over $K$, \textit{i.e.} comes from an algebraic group $R/K$, then the \emph{semisimple quotient} of $H$ is $H^{ss}:=H/R$ — in that case $H/R$ is a semisimple $K$-group.

\begin{thm}[{\cite[Th. 5.1.3]{ConradPrasad_classification}}]
\label{thm revetement}
Let $G$ be a perfect smooth connected linear algebraic $k$-group. Take $K/k$ to be the definition field of the unipotent radical $R_u(G_{\bar{k}})$ of $G_{\bar{k}}$.

Then, for any tame central extension \[(E) \: \: \: 1 \to \nu \to H \to G \to 1 , \] where $H$ is a perfect smooth connected linear $k$-group, the minimal field of definition of $R_u(H_{\bar{k}})$ is $K$.

Also, there exists a tame central extension \[ \: (E_0) \: \: 1 \to \mu \to \widetilde{G} \to G \to 1 , \] where $\tilde{G}$ is a perfect smooth connected linear algebraic $k$-group and this extension satifies the following property : for any other tame central extension ${E}$ \[(E) \: \: \: 1 \to \nu \to H \to G \to 1 , \] where $H$ is a perfect smooth connected linear $k$-group, the semisimple group $(\widetilde{G}_K)^{ss}$ is a covering of $(H_K)^{ss}$ and the set of (isomorphism classes of) morphisms ${E}_0 \to {E}$ is in a one-to-one correspondence via the obvious map with the set of homomorphisms $(\widetilde{G}_K)^{ss} \to (H_K)^{ss}$.

The extension $E_0$ is characterized by the fact that $\widetilde{G}_K^{ss}$ is the universal covering of the semisimple group $G_K^{ss}$.
\end{thm}

The extension $E_0$ (or just the group $\widetilde{G}$) is called the \emph{universal tame central extension/covering} of $G$.

\begin{ex}
Let $k^\prime/k$ be a finite, purely inseparable field extension. Let $G^\prime$ be a semisimple, simply connected $k^\prime$-group and let $\mu^\prime$ be a central subgroup of $G^\prime$. The $k$-group $G:= \mathrm{R}_{k^\prime/k}(G^\prime)/\mathrm{R}_{k^\prime/k}(\mu^\prime)$ admits an extension \[({E}_0) \: \: \: 1 \to \res (\mu^\prime) \to \res(G^\prime) \to G \to 1. \] and the following holds : \begin{itemize}
\item $\res (\mu^\prime)$ is tame and central in $\res(G^\prime)$ (for tameness, see the discussion following \cite[Def. 5.1.1]{ConradPrasad_classification});
\item $\res(G^\prime)$ and $G$ are smooth, connected and perfect algebraic groups (\cite[Prop. 1.1.10 \& 1.3.4]{CGP}) - actually they are pseudo-semisimple.
\end{itemize} Then, since $(\res (G^\prime)_{\bar{k}})^{ss}=G^\prime_{\bar{k}}$ is simply connected, $E_0$ must be the universal tame central extension of $G$.
\end{ex}

		\subsection{Extension groups of perfect groups}

Let $G$ be a perfect smooth connected linear algebraic $k$-group. Theorem \ref{thm revetement} yields the universal tame central extension \[ ({E}_0) \: \: 1 \to \mu \to \widetilde{G} \to G \to 1  \] where $\widetilde{G}$ is a perfect connected smooth linear algebraic $k$-group. As for semisimple groups, one defines a homomorphism $\Theta : \widehat{\mu}(k) \to \mathrm{Ext}^1(G,\mathbb{G}_m)$ as follows : for any group homomorphism $\chi : \mu \to \mathbb{G}_m$, consider the pushforward of $E_0$ with respect to $\chi$, that is the extension \[ ({E}_\chi) \: \: \: 1 \to \mathbb{G}_m \to H \to G \to 1 \] where $H$ is the quotient of $\widetilde{G} \times_k \mathbb{G}_m$ by $Z_\chi:=\{ (x, \chi(x)^{-1}) \; | \; x \in \mu \}$. There is a commutative diagram \begin{equation}
\label{diag poussée}
\xymatrix{
 1 \ar[r] & \mu \ar[r] \ar[d]^{\chi} & \widetilde{G} \ar[d] \ar[r] & G \ar[d]^{=} \ar[r] & 1 \\
 1 \ar[r] & \mathbb{G}_m \ar[r] & H \ar[r] & G \ar[r] & 1
} . \end{equation} The map $\Theta$ is then defined to be the map that sends $\chi \in \widehat{\mu}(k)$ to the isomorphism class $\mathcal{E}_\chi$ of ${E}_\chi$. The map $\Theta$ is easily seen to be a group homomorphism.

\begin{prop}
\label{prop groupe extensions pseudo ss}
Let $G$ be a perfect smooth connected linear algebraic $k$-group. Denote the kernel of its universal tame central extension by $\mu$. Then the homomorphism $\Theta : \widehat{\mu}(k) \to \mathrm{Ext}^1(G,\mathbb{G}_m)$ is an isomorphism.
\end{prop}

\begin{proof}
By \cite[Lem. 3.3]{Rosengarten_picard} one has an exact sequence \[ \widehat{\widetilde{G}}(k) \to \widehat{\mu}(k) \overset{\Theta}{\to} \ext(G,\mathbb{G}_m) \to \ext(\widetilde{G},\mathbb{G}_m)  . \] But since $\widetilde{G}$ is perfect, it has no non trivial character ; and by \cite[Lem. 5.3]{Rosengarten_Tamagawa}, we know that $\ext(\tilde{G},\mathbb{G}_m)$ is trivial.
\end{proof}

\begin{rem}
For $G$ as in Proposition \ref{prop groupe extensions pseudo ss}, the canonical inclusion $\mathrm{Ext}^1(G,\mathbb{G}_m) \hookrightarrow \pic(G)$ is an isomorphism, so $\pic(G) \simeq \widehat{\mu}(k)$. This follows either by Proposition \ref{prop additivite picard} and Theorem \ref{prop groupe retract rationnel}, or by \cite[Prop. 3.2]{Rosengarten_picard}.
\end{rem}

		\subsection{Application to some pseudo-semisimple groups}

Pseudo-semisimple groups defined as quotients $\res(G^\prime)/\res(\mu^\prime)$ (where $k^\prime/k$ is a finite, non-trivial purely inseparable extension, $G^\prime/k^\prime$ is a non-trivial semisimple and simply connected and $\mu^\prime \subset G^\prime$ is central) form a wide family of non-reductive groups. We have seen that the kernels of their universal tame extensions are of the form $\res(\mu^\prime)$ for groups $\mu^\prime$ which are finite and of multiplicative type. But thanks to Theorem \ref{thm caracteres restriction weil}, the character groups of such a $\res(\mu^\prime)$ is known when $\mu^\prime$ is \'etale ou diagonalisable, and we get the following proposition.

\begin{prop}
\label{prop inv groupes pss}
Let $k^\prime/k$ be a finite, purely inseparable field extension. Write $p$ for the characteristic exponent of $k$ and define $h$ to be the least non-negative integer such that $(k^\prime)^{p^h} \subseteq k$. Let $G^\prime$ be a semisimple, simply connected $k^\prime$-group and let $\mu^\prime$ be a central subgroup of $G^\prime$. Write $G=\res(G^\prime)/ \res(\mu^\prime)$. If $\mu^\prime$ is diagonalisable, then there is an isomorphism  \[ \ext(G,\mathbb{G}_m) \simeq p^h \widehat{\mu^\prime}(k^\prime) . \] If $\mu^\prime$ is \'etale and of multiplicative type, the isomorphism still holds but we have simply $p^h \widehat{\mu^\prime}(k^\prime) = \widehat{\mu^\prime}(k^\prime)$.
\hfill \qed
\end{prop}

\begin{ex}
Let $m$ be a non negative integer and let $k^\prime/k$ be a finite, purely inseparable extension. Recall the notation $p$ and $h$ from Proposition \ref{prop inv groupes pss}. One has a group isomorphism : \[ \ext(\res(\mathrm{SL}_{p^m,k^\prime})/\res(\mu_{p^m,k^\prime}), \mathbb{G}_m) \simeq p^h \cdot \mathbb{Z}/p^m \mathbb{Z} , \] where $\mu_{p^m,k^\prime}$ is the group of $p^m$-th roots of unity inside the special linear group $\mathrm{SL}_{p^m,k^\prime}/k^\prime$. Notice that if $h \geqslant m$, then the extension group is trivial.
\end{ex}


		\subsection{Invariance by scalar extensions}

For a reductive group $G/k$, by \cite[Prop. 3.3]{Colliot_resolutions} there exist two $k$-tori $T$ and $S$ such that \[ \widehat{T}(L) \to \widehat{S}(L) \to \pic(G_L) \to 0 \] for all field extensions $L/k$. It is then straighforward to see that the Picard group of $G$ is invariant under a purely inseparable field extension, and the same invariance property holds for any field extension when $G$ is moreover split. Here is an analogue of the latter statement for standard pseudo-semisimple groups.

In the sequel we will focus on \emph{standard} pseudo-semisimple groups. Note that all pseudo-reductive groups are standard in characteristic $p>3$ (\cite[Th.1.5.1]{CGP}). The standardness hypothesis is used here to have a concrete grip on the studied groups.

We recall the definition of standardness from \cite[\S 1.4]{CGP}. Let $G$ be a pseudo-reductive $k$-group. The group $G$ is called {standard} if it is isomorphic to a central quotient \[ \frac{ \left( \prod_{i=1}^{r} \mathrm{R}_{k_i^\prime/k}(G_i^\prime) \right) \rtimes C}{\prod_{i=1}^{r} \mathrm{R}_{k_i^\prime/k}(T_i^\prime)} \] for \emph{standard data} $\mathcal{D}=\left( (k_i^\prime)_i,(G_i^\prime)_i,(T_i^\prime)_i,C,\phi,\psi \right)$ where : \begin{itemize}
\item $k_1^\prime, \cdots,k_r^\prime$ are finite field extensions  of $k$;
\item $G_i^\prime/k_i^\prime$ are connected reductive groups and $T_i^\prime \subset G_i^\prime$ maximal $k_i^\prime$-tori;
\item $C$ is a commutative pseudo-reductive group which fits into a factorization \[ \prod_{i=1}^r \mathrm{R}_{k_i^\prime/k}(T_i^\prime) \overset{\phi}{\longrightarrow} C \overset{\psi}{\longrightarrow} \prod_{i=1}^r \mathrm{R}_{k_i^\prime/k}(T_i^\prime/Z_i^\prime) \] of the product of the canonical maps $\mathrm{R}_{k_i^\prime/k}(T_i^\prime) \to \mathrm{R}_{k_i^\prime/k}(T_i^\prime/Z_i^\prime)$ where $Z_i^\prime$ stands for the center of $G_i^\prime$;
\item the group $C$ acts on $\prod_{i=1}^r \mathrm{R}_{k_i^\prime/k}(G_i^\prime)$ by conjugation via $\phi$;
\item and finally, $\prod_{i=1}^r \mathrm{R}_{k_i^\prime/k}(T_i^\prime)$ is seen as a central subgroup of the semi-direct product through \[ t=(t_i)_i \mapsto \left( t, \phi(t)^{-1} \right) . \]
\end{itemize} 

We have :

\begin{thm}
\label{prop rigidite pss}
Let $G$ be a standard pseudo-semisimple $k$-group which is pseudo-split. For any separable extension $K/k$ (finite or not), the scalar extension homomorphism $\ext(G,\mathbb{G}_m) \to \ext(G_K,\mathbb{G}_m)$ is an isomorphism.
\end{thm}

In order to prove Theorem \ref{prop rigidite pss}, we discuss first some properties of the standard data $\mathcal{D}$ of a given standard pseudo-semisimple group $G$ which is pseudo-split. The following paragraphs and argument are due to the anonymous referee.

First, there exist standard data $\mathcal{D}$ such that $G_i^\prime$ are absolutely simple and simply connected (this is \cite[Th. 4.1.1]{CGP}). We write $\widetilde{G}$ for $\prod_{i=1}^{r} \mathrm{R}_{k_i^\prime/k}(G_i^\prime)$. With such $G_i^\prime$'s the homomorphism \[ \pi : \widetilde{G}  \to \frac{ \left( \prod_{i=1}^{r} \mathrm{R}_{k_i^\prime/k}(G_i^\prime) \right) \rtimes C}{\prod_{i=1}^{r} \mathrm{R}_{k_i^\prime/k}(T_i^\prime)} \simeq G \] is surjective by \cite[Prop. 4.1.4(1)]{CGP} and since $\prod_{i=1}^{r} \mathrm{R}_{k_i^\prime/k}(T_i^\prime)$ is central in the semi-direct product, $\pi$ has central kernel $Z$, so $Z$ is in the center of $\widetilde{G}$ which is $\prod_{i=1}^r \mathrm{R}_{k_i^\prime/k}(Z_i^\prime)$ by \cite[Prop. A.5.15(1)]{CGP}.

Then, take a $k$-split maximal $k$-torus $S$ in $G$. Restricting $\pi$ to $\pi^{-1}(S)$ yields a surjective homomorphism $\pi^{-1}(S) \to S$. Taking a maximal $k$-torus $\widetilde{S} \subset \pi^{-1}(S)$, its image by $\pi$ is exactly $S$ because of \cite[Prop. A.2.8]{CGP}. Moreover, we see that any maximal torus of $\widetilde{G}$ containing $\widetilde{S}$ is mapped onto a torus containing $S$, so its image is exactly $S$ and it follows that $\widetilde{S}$ must be maximal in $\widetilde{G}$. By \cite[Prop. A.5.12(2)]{CGP}, there are maximal tori $S_i^\prime$ of $G_i^\prime$ over $k_i^\prime$ such that $\widetilde{S}$ is the maximal $k$-torus of $\prod_{i=1}^r \mathrm{R}_{k_i^\prime/k}(S_i^\prime)$. It is possible to describe $\widetilde{S}$ as follows. Define $k_i$ to be the separable closure of $k$ in $k_i^\prime$. The extension $k_i^\prime/k_i$ is finite and purely inseparable, so $S_i^\prime$ descents to a torus $S_i$ over $k_i$. We notice that the quotient $\mathrm{R}_{k_i^\prime/k}(S_i^\prime)/\mathrm{R}_{k_i/k}(S_i)$ is unipotent; in fact, it is seen to be isomorphic to $\mathrm{R}_{k_i/k} \left( \mathrm{R}_{k_i^\prime/k_i}(S_i^\prime) / S_i \right)$ by \cite[Cor. A.5.4(3)]{CGP} but the group $\mathrm{R}_{k_i^\prime/k_i}(S_i^\prime) / S_i$ is unipotent (see the proof of Lemma \ref{csq quotient unipotent}) and the Weil restriction of a unipotent group is still unipotent (this is \cite[Prop. A.3.7]{Oesterle}). The group $\mathrm{R}_{k_i/k}(S_i)$ is then the maximal torus of the abelian group $\mathrm{R}_{k_i^\prime/k}(S_i^\prime)$ and the torus $\widetilde{S}$ is equal to the product of the many $\mathrm{R}_{k_i/k}(S_i)$'s. We now claim that $\pi : \widetilde{S} \to S$ is an isogeny. To see this, we proceed as before : the group $Z_i^\prime$ descents to $Z_i$ over $k_i$ and $\mathrm{R}_{k_i/k}(Z_i)$ is the maximal subgroup of multiplicative type of $\mathrm{R}_{k_i^\prime/k}(Z_i^\prime)$ (here, since $Z_i$ may be non-smooth, we just know that the quotient $\mathrm{R}_{k_i^\prime/k}(Z_i^\prime)/\mathrm{R}_{k_i/k}(Z_i)$ is a \textit{subgroup} of the Weil restriction of the unipotent group $\mathrm{R}_{k_i^\prime/k_i} (Z_i^\prime)/Z_i$). Note that $\mathrm{R}_{k_i/k}(Z_i)$ is finite. Thus, the restriction $\widetilde{S} \to S$ of $\pi$ has kernel $\widetilde{S} \cap Z \subset \prod_{i=1}^r \mathrm{R}_{k_i/k}(Z_i)$, and it follows that $\widetilde{S} \to S$ is an isogeny. This implies that $\widetilde{S}$ is split because $S$ is split. We derive two consequences of interest. The first is that $S_i^\prime$, and hence $G_i^\prime$, are split by \cite[Prop. A.5.15(2)]{CGP}. The second is that the field extensions $k_i^\prime/k$ are purely inseparable. Indeed, the maximal torus $\widetilde{S}$ of $\widetilde{G}$ is split, but we have $\widetilde{S}=\prod_i \mathrm{R}_{k_i/k}(S_i)$ and the torus $\mathrm{R}_{k_i/k}(S_i)$ is $k$-split if, and only if, $k_i=k$.

\begin{proof}[Proof of Theorem \ref{prop rigidite pss}]

Let $G$ be a standard pseudo-semisimple $k$-group which is pseudo-split. Take $\mathcal{D}=\left( (k_i^\prime)_i,(G_i^\prime)_i,(T_i^\prime)_i,C,\phi,\psi \right)$ to be standard data for $G$ with all $G_i^\prime$ absolutely simple, simply connected, and split. Writing $\widetilde{G}$ for \[ \mathrm{R}_{k_1^\prime/k}(G_1^\prime) \times_k \cdots \times_k \mathrm{R}_{k_r^\prime/k}(G_r^\prime), \] the group $G$ is the quotient of $\widetilde{G}$ by a central subgroup $Z$. The group $Z$ is a subgroup of $ \mathrm{R}_{k_1^\prime/k}(Z_1^\prime) \times_k \cdots \times_k \mathrm{R}_{k_r^\prime/k}(Z_r^\prime)$, where $Z_i^\prime$ is the center of $G_i^\prime$. The $Z_i^\prime/k_i^\prime$ are finite of multiplicative type, so $Z$ is a tame $k$-group. Thus \[ 1 \to Z \to \widetilde{G} \to G \to 1 \] is the universal tame extension of $G$ and \[ 1 \to Z_K \to \widetilde{G}_K \to G_K \to 1 \] is that of $G_K$ ($K/k$ is the separable extension given in the statement of the theorem). Let's denote by $\phi$ the homomorphism $\widehat{Z}(k) \to \ext(G,\mathbb{G}_m)$ from Proposition \ref{prop groupe extensions pseudo ss} for $G$, and $\phi_K$ the similar homomorphism for $G_K$. There is a commutative diagram \[ \xymatrix{
\widehat{Z}(k) \ar[r]^-\phi \ar[d] & \ext(G,\mathbb{G}_m) \ar[d] \\
\widehat{Z}(K) \ar[r]_-{\phi_K} & \ext(G_K,\mathbb{G}_m)
} \] where the vertical arrows are obtained by scalar extension and the horizontal ones are isomorphisms by Proposition \ref{prop groupe extensions pseudo ss}. So it remains to show that $\widehat{Z}(k) \to \widehat{Z}(K)$ is an isomorphism.

The groups $Z_i^\prime$ are defined over $k$ : $Z_i^\prime=(Z_i)_{k_i^\prime}$ for a finite diagonalisable $k$-group $Z_i$. Write \[ Z_0:= Z_1 \times_k \cdots \times_k Z_r \text{ and } Z_0^\dagger:=\mathrm{R}_{k_1^\prime/k}(Z_1^\prime) \times_k \cdots \times_k \mathrm{R}_{k_r^\prime/k}(Z_r^\prime) . \] As explained in the paragraph before the proof or in the proof of Lemma \ref{csq quotient unipotent}, $Z_0$ is the maximal subgroup of multiplicative type inside $Z_0^\dagger$ and the quotient $Z_0^\dagger /Z_0$ is unipotent. Since unipotent groups have only trivial character (\cite[Exp. XVII, Prop. 2.4]{SGA3}), the restriction homomorphism $\widehat{Z}(L) \to \widehat{(Z \cap Z_0)}(L)$ is injective for all field extensions $L/k$. We will see now that this implies that $\widehat{Z}(k) \to \widehat{Z}(K)$ is a bijection. Note that since $Z \cap Z_0$ is diagonalizable, $\widehat{(Z \cap Z_0)}(k)=\widehat{(Z \cap Z_0)}(L)$ for an arbitrary extension $L/k$. \begin{itemize}
\item  If $K/k$ is separable algebraic, then $\widehat{Z}(k)=\widehat{Z}(k_s)^{\mathrm{Gal}(k_s/k)}$ and $\widehat{Z}(K)=\widehat{Z}(k_s)^{\mathrm{Gal}(k_s/K)}$. But the injection $\widehat{Z}(k_s) \to \widehat{(Z \cap Z_0)}(k_s)$ tells us that the Galois actions are trivial on $\widehat{Z}(k_s)$. Thus, $\widehat{Z}(k)=\widehat{Z}(K)$ as wanted.
\item If $K=k(t)$ is a purely transcendental extension, one can let $\mathrm{Aut}(K/k)$ act on 
$\widehat{Z}(K)$ and the fixed points are exactly $\widehat{Z}(k)$. But as in the previous point, we know that the action of $\mathrm{Aut}(K/k)$ is trivial on $\widehat{(Z \cap Z_0)}(K)$ and on $\widehat{Z}(k) \subseteq \widehat{(Z \cap Z_0)}(K)$.
\item By combining the two latter points, one sees that $\widehat{Z}(k)=\widehat{Z}(K)$ for any finitely generated separable extension $K/k$.
\item It remains to treat the case of a general separable extension $K/k$. By definition this means that for any finitely generated extension $k \subseteq K_0 \subseteq K$, $K_0/k$ is separably generated. One gets that $\widehat{Z}(K)=\underset{\underset{k \subseteq K_0 \subseteq K}{\longrightarrow}}{\mathrm{lim}} \widehat{Z}(K_0)$, the limit being over finitely generated subextensions of $K/k$. Since by the previous point one has $\widehat{Z}(k)=\widehat{Z}(K_0)$ for any $K_0 \subseteq K$ finitely generated over $k$, one is lead to $\widehat{Z}(k)=\widehat{Z}(K)$ as wanted.
\end{itemize}

In conclusion, the scalar extension homomorphism $\ext(G,\mathbb{G}_m) \to \ext(G_K,\mathbb{G}_m)$ is indeed an isomorphism.
\end{proof}

\appendix

	\section{Character groups of Weil restrictions}
	\label{section caracteres weil}

Let $k^\prime/k$ be a finite, purely inseparable field extension and let $G/k$ be a linear algebraic group. We describe here the character group of $\res(G_{k^\prime})$ in terms of the character group of $G$, for algebraic groups $G$ either smooth and of multiplicative type, or diagonalisable. As an easy byproduct, we also treat the case of a connected reductive group.

\begin{nota}
Call $p$ the characteristic exponent of $k$ and let $h$ be the minimal non-negative integer such that $(k^\prime)^{p^h} \subseteq k$. Also write $G_0$ for $\res(G_{k^\prime})$ and recall that $G$ can be embedded canonically as a closed $k$-subgroup of $G_0$. We recall that for an algebraic $k$-group $H$, $\widehat{H}(k)$ stands for the group of $k$-characters of $H$.
\end{nota}

\begin{lemma}
\label{csq quotient unipotent}
Let $G$ be either smooth and connected or of multiplicative type. Then the restriction homomorphism $\widehat{G_0}(k) \to \widehat{G}(k)$ is injective.
\end{lemma}

\begin{proof} $\diamond$ First assume that $G$ is smooth and connected. According to \cite[Prop. A.5.11(2)]{CGP}, there is a canonical surjective map $(G_0)_{k^\prime} \to G_{k^\prime}$ whose kernel is unipotent; working with the functors of points, we easily check that the embedding $G_{k^\prime} \to (G_0)_{k^\prime}$ is a section of the projection $(G_0)_{k^\prime} \to G_{k^\prime}$; thus $(G_0)_{k^\prime}$ is a semi-direct product of a unipotent $k^\prime$-group $U^\prime$ by $G_{k^\prime}$. Thus, if $\chi \in G_0(k)$ becomes trivial when restricted to $G$, then $\chi_{k^\prime}$ is constant equal to $1$ on $(G_0)_{k^\prime}= U^\prime \rtimes G_{k^\prime}$ because it is trivial on $G_{k^\prime}$ by assumption and unipotent groups have only trivial characters (see \cite[Exp. XVII, Prop. 2.4]{SGA3}). So $\chi_{k ^\prime}$ is the trivial character, hence so is $\chi$.

$\diamond$ In the case the group $G$ is of multiplicative type, it becomes diagonalisable over $k_s$ and can be embedded in a split torus $\mathbb{G}_{m,k_s}^N$. This way, denoting by $k_s^\prime$ the field $k^\prime \otimes_k k_s$, the group $(G_0/G)_{k_s}$ is seen to be a subgroup of $\mathrm{R}_{k_s^\prime/k_s}(\mathbb{G}_{m,k_s^\prime}^N) / \mathbb{G}_{m,k_s}^N$ which is unipotent by \cite[Prop. A.5.11(2)]{CGP}. It follows that $\widehat{G_0}(k) \to \widehat{G}(k)$ is injective, because if $\chi \in \widehat{G_0}(k)$ is trivial on the subgroup $G \hookrightarrow G_0$, then $\chi$ yields a homomorphism  $\widetilde{\chi} : G_0/G \to \mathbb{G}_m$ and since $G_0/G$ is unipotent, \cite[Exp. XVII, Prop. 2.4]{SGA3} implies that $\widetilde{\chi}$ is trivial. Thus $\chi$ is trivial, whence the result.
\end{proof}

\begin{const}
\label{construction caracteres} For $G$ as in the lemma, one has a homomorphism $\widehat{G}(k) \to \widehat{G}(k^\prime)$ simply given by scalar extension. Also, for $\chi \in \widehat{G}(k^\prime)$, one defines a character of $G_0$ whose expression on $A$-points (for any $k$-algebra $A$) is \[ \begin{array}{ccccc}
G_0(A)=G(A \otimes_k k^\prime) & \overset{\chi}{\longrightarrow} & (A \otimes_k k^\prime)^\ast & {\longrightarrow} & A^\ast \\
g & \longmapsto & \chi(g) & \longmapsto & \chi(g) ^{p^h}
\end{array} \] where the last arrow is the $p^h$-th power map. This defines a homomorphism \[\tau : \widehat{G}(k) \to \widehat{G_0}(k) \] whose composition with the injective restriction homomorphism $\widehat{G_0}(k) \hookrightarrow \widehat{G}(k)$ is $\chi \mapsto (\chi)^{p^h}$. It follows that the subgroup $\widehat{G_0}(k)$ of $\widehat{G}(k)$ contains $p^h \widehat{G}(k)$. We will see that in several cases, $\tau$ is surjective and $\widehat{G_0}(k)$ is the subgroup $p^h \widehat{G}(k)$ of $\widehat{G}(k)$.

\end{const}

\begin{thm}
\label{thm caracteres restriction weil}
Consider the following cases \begin{itemize}
\item[(i)] $G$ is diagonalisable (not necessarily smooth or connected),
\item[(ii)] $G$ is smooth and of multiplicative type,
\item[(iii)] $G$ is connected and reductive.
\end{itemize} Then the restriction homomorphism $\widehat{G_0}(k) \to \widehat{G}(k)$ is injective and identifies $\widehat{G_0}(k)$ with $p^h \widehat{G}(k) \subseteq \widehat{G}(k)$.
\end{thm}

\begin{rem}
\label{remarque caracteres}\begin{enumerate}
\item Theorem \ref{thm caracteres restriction weil} is an analogue of Oesterl\'e's \cite[Th. II.2.4]{Oesterle} for Weil restrictions through purely inseparable field extensions. 
\item For finite \'etale groups $G$ of multiplicative type, the statement is trivial. Indeed, the embedding $G \to G_0$ is an isomorphism by \cite[Prop. A.5.13]{CGP} and $p^h \widehat{G}(k) \subseteq \widehat{G_0}(k) \subseteq \widehat{G}(k)$ are equalities since $p^h$ is prime to the order of $G$. 
\end{enumerate} 
\end{rem}

\begin{center}
\underline{\textit{Proof of Theorem \ref{thm caracteres restriction weil} - The beginning}}
\end{center}

\paragraph{\textit{$\bullet$ The case $G=\mathbb{G}_m$.}} Here we treat the case when the algebraic $k$-group $G$ is the multiplicative group $\mathbb{G}_m$.

There is an identification $\mathbb{Z} \overset{\sim}{\to} \widehat{\mathbb{G}_m}(k)$, $r \mapsto (x \mapsto x^r)$. We want to show that the image of the one-to-one homomorphism $R : \widehat{\res(\mathbb{G}_{m,k^\prime})}(k) \hookrightarrow \widehat{\mathbb{G}_m}(k)$ is exactly the subgroup $p^h \mathbb{Z}$.

First of all, thanks to Construction \ref{construction caracteres}, the integer $p^h$ is in the image of $R$.

Secondly, $\widehat{\res(\mathbb{G}_{m,k^\prime})}(k)$ is a subgroup of $\mathbb{Z}$ and has a generator $\rho$. The integer $r=R(\rho)$ is non-zero and we may assume that it is positive.

To show that $r=p^h$, write $r=p^\alpha s$ for some positive integer $s$ prime to $p$. Since $p^h \in r \mathbb{Z}= \mathrm{Im}(R)$, $s$ must be $1$ and $\alpha \leqslant h$. Thus, both homomorphisms ${(\cdot)^{p^h}, \; \rho^{p^{h-\alpha}} : \res(\mathbb{G}_m) \to \mathbb{G}_m}$ become equal when restricted to $\mathbb{G}_m \subset \res(\mathbb{G}_m)$, so they are equal. Looking at $k$-points we find \[ \forall x \in (k^\prime)^\times, \, x^{p^h}=\rho(x)^{p^{h-\alpha}}. \] But the Frobenius homomorphism is injective, so \[ \forall x \in (k^\prime)^\times, x^{p^\alpha} = \rho(x) .\] This shows that $({k^\prime}^\times)^{p^\alpha}=\rho({k^\prime}^\times) \subseteq k^\times$, whence $\alpha=h$.

In conclusion, Theorem \ref{thm caracteres restriction weil} is verified for the multiplicative group.

\paragraph{\textit{$\bullet$ The case $G=\mu_n$.}} We treat the case when $G=\mu_n$, the algebraic $k$-group of $n$-th roots of unity for some positive integer $n$. The group $\mu_n$ is a product of $\mu_{n_0}$ and $\mu_{p^{n_1}}$ where $n_0$ is prime to $p$ and $n_1$ is a non negative integer. We have thus to treat two cases : the case $n$ is prime to $p$ and $n$ is a $p$-power. Similarly to the previous point, there is an identification $\mathbb{Z}/n \mathbb{Z} \overset{\sim}{\to} \widehat{\mu_n}(k)$, $r \mapsto (x \mapsto x^r)$. We want to show that the image of the one-to-one homomorphism $R : \widehat{\res(\mu_n)}(k) \hookrightarrow \widehat{\mu_n}(k)$ is exactly the subgroup $p^h \mathbb{Z}/n \mathbb{Z}$. The case when $n$ is prime to $p$ is trivial as stated in Remark \ref{remarque caracteres}. So assume $n$ is equal to $p^m$ for some integer $m$.

From Construction \ref{construction caracteres}, we already know that the image of $R$ contains $p^h \widehat{G}(k)$. For the other inclusion, consider the commutative diagram with exact rows \[ \xymatrix{
1 \ar[r] & \res(\mu_{p^m,k^\prime}) \ar[r] & \res(\mathbb{G}_{m,k^\prime}) \ar[r] & \mathrm{R}_{k({k^\prime}^{p^m})/k}(\mathbb{G}_{m,k({k^\prime}^{p^m})}) \ar[r] & 1 \\
1 \ar[r] & \mu_{p^m} \ar[r] \ar[u] & \mathbb{G}_m \ar[r] \ar[u] & \mathbb{G}_m \ar[r] \ar[u] & 1
} \] where the surjective arrows are $p^m$-th power maps (the first row is exact by Lemma \ref{suite}. Applying \cite[Lem. 3.3]{Rosengarten_picard} and using the fact that $\mathrm{R}_{k({k^\prime}^{p^m})/k}(\mathbb{G}_{m,k({k^\prime}^{p^m})})$ has a trivial Picard group, lead to a commutative square \[ \xymatrix{
\widehat{\res(\mathbb{G}_{m,k^\prime})}(k) \ar@{->>}[r] \ar@{^{(}->}[d] & \widehat{G}_0(k) \ar@{^{(}->}[d]  \\
\widehat{\mathbb{G}_m}(k) \ar[r] & \widehat{G}(k)
} . \] Since Theorem \ref{thm caracteres restriction weil} is known for $\mathbb{G}_m$, this square tells us that that $\mathrm{Im}(R)$ is contained in $p^h \widehat{G}(k)$.

So Theorem \ref{thm caracteres restriction weil} holds for groups of roots of unity $\mu_n$.

\paragraph{\textit{$\bullet$ The case $G$ is diagonalisable.}} If $G$ is a diagonalisable $k$-group, then $G$ is a product of finitely many $\mathbb{G}_m$'s and $\mu_n$'s (groups of roots of unity). Theorem \ref{thm caracteres restriction weil} has been shown for $\mathbb{G}_m$ and for all $\mu_n$, whence the result for the diagonalisable $G$.

\paragraph{\textit{$\bullet$ The case $G$ is smooth of multiplicative type.}} Every smooth algebraic $k$-group $G$ of multiplicative type is diagonalisable over $k_s$. Thus, according to what has been done above, the image of \begin{equation} 
\label{morphisme sep} \widehat{G_0}(k_s) \to \widehat{G}(k_s) 
\end{equation} is $p^h \widehat{G}(k_s)$. Restricting the injective homomorphism \eqref{morphisme sep} to elements that are invariant under the absolute Galois group of $k$, we find that the image of $\widehat{G_0}(k) \to \widehat{G}(k)$ is $\left( \widehat{G}(k_s) \right) ^\Gamma \bigcap \left( p^h \widehat{G}(k_s) \right) $. Since $G$ is assumed to be smooth, $\widehat{G}(k_s)$ is a product of copies of $\mathbb{Z}$ and of groups $\mathbb{Z}/n\mathbb{Z}$ for integers $n$ coprime to $p$; thus the $p^h$-th power homomorphism of $\widehat{G}(k_s)$ is injective, and this implies that $ \widehat{G}(k) \bigcap \left( p^h \widehat{G}(k_s) \right) $ is equal to $p^h \widehat{G}(k)$ inside $\widehat{G}(k)$.

\paragraph{\textit{$\bullet$ The case $G$ is reductive and connected.}} The connected reductive group $G$ fits into an exact sequence of algebraic groups \begin{equation}
\label{balboa}
1 \to H \to G \to T \to 1
\end{equation} where $H$ is the semisimple derived group of $G$ and $T$ is a torus. The sequence \eqref{balboa} induces an exact sequence \[ 1 \to \widehat{T}(k) \to \widehat{G}(k) \to \widehat{H}(k).\] Since the group $H$ is perfect, $\widehat{H}(k)$ is trivial and one finds that the restriction homomorphism $\widehat{T}(k) \to \widehat{G}(k)$ is an isomorphism.

Also, applying \cite[Cor. A.5.4(3)]{CGP} to the sequence \eqref{balboa} yields an exact sequence \[
1 \to H_0 \to G_0 \to T_0 \to 1 \] where $G_0=\res(G_{k^\prime})$, $T_0=\res(T_{k^\prime})$ and $G_0=\res(G_{k^\prime})$. This induces an exact sequence \[ 1 \to \widehat{T_0}(k) \to \widehat{G_0}(k) \to \widehat{H_0}(k)  \] but $\widehat{H_0}(k)$ is trivial because extending scalars to $k^\prime$, $(H_0)_{k^\prime}$ becomes a semi-direct product of a unipotent $k^\prime$-group by the perfect group $H_{k^\prime}$. So $\widehat{T_0}(k) \simeq \widehat{G_0}(k)$.

Theorem \ref{thm caracteres restriction weil} applied to $T$ then tells us that the restriction homomorphism $\widehat{G_0}(k) \to \widehat{G}(k)$ is injective and identifies $\widehat{G_0}(k)$ with $p^h \widehat{G}(k) \subseteq \widehat{G}(k)$ as desired.

\begin{center}
\underline{\textit{Proof of Theorem \ref{thm caracteres restriction weil} - The end} \qed}
\end{center}


\end{document}